\documentclass[11pt]{article}
\usepackage{amsmath,amssymb,amsfonts,amsthm,graphicx}
\usepackage{hyperref}
\usepackage[titletoc,title]{appendix}
\newtheorem{theorem}{Theorem}[section]
\newtheorem{proposition}[theorem]{Proposition}
\newtheorem{corollary}[theorem]{Corollary}

\newtheorem{remark}[theorem]{Remark}

\numberwithin{equation}{section}
\title{Convergence rate for R\'enyi-type continued fraction expansions}
\author{
    Gabriela Ileana Sebe\footnote{e-mail: igsebe@yahoo.com.} \\
    \emph{\small Politehnica University of Bucharest, Faculty of Applied Sciences},\\
    \emph{\small Splaiul Independentei 313, 060042, Bucharest, Romania} and \\
    \emph{\small Institute of Mathematical Statistics and Applied Mathematics}, \\
     \emph{\small Calea 13 Sept. 13, 050711 Bucharest, Romania} \\
    and\\
    Dan Lascu\footnote{e-mail: lascudan@gmail.com.}\nonumber \\
    \emph{\small Mircea cel Batran Naval Academy, 1 Fulgerului, 900218 Constanta,
    Romania} \\
    }
\begin{document}
\maketitle
\thispagestyle{empty}
\begin{abstract}
This paper continues our investigation of Renyi-type continued fractions studied in \cite{Sebe&Lascu-2018}.
A Wirsing-type approach to the Perron-Frobenius operator of the R\'enyi-type continued fraction transformation under its invariant measure
allows us to study the optimality of the convergence rate.
Actually, we obtain upper and lower bounds of the convergence rate which provide a near-optimal solution to the Gauss-Kuzmin-L\'evy problem.
\end{abstract}

\section{Introduction}

In \cite{Sebe&Lascu-2018} we showed that every $x \in [0, 1]$ can be written as
\begin{equation} \label{1.1}
x = 1 - \displaystyle \frac{N}{1+a_1 - \displaystyle \frac{N}{1+a_2 - \displaystyle \frac{N}{1+a_3 - \ddots}}} :=[a_1, a_2, a_3, \ldots]_R,
\end{equation}
where $N$ is a fixed integer greater than or equal to $2$, and $a_n$'s are positive integers greater than or equal to $N$.

For a fixed integer $N \geq 2$, we define the shift transformation $R_N : [0,1] \rightarrow [0,1]$, by $R_N (1) = 0$ and $R_N(x)=R_N([a_1, a_2, a_3, \ldots]_R) := [a_2, a_3, a_4, \ldots]_R$ for $x \neq 1$.
There is another way to define $R_N$ for $x \in [0, 1]$. With $\lfloor \cdot \rfloor$ denoting the floor function, we have
\begin{equation}
R_{N}(x) :=
\left\{
\begin{array}{ll}
{\displaystyle \frac{N}{1-x}- \left\lfloor\frac{N}{1-x}\right\rfloor},&
{ x \in [0, 1) }\\
0,& x=1.
\end{array}
\right. \label{1.2}
\end{equation}
Since the case $N=1$ refers to the R\'enyi interval map,
we call the transformation in (\ref{1.2}) \textit{R\'enyi-type continued fraction transformation}.

The digits or incomplete quotients of $x$ with respect to the R\'enyi-type continued fraction expansion are defined by
\begin{equation}
a_n := a_n(x) = a_1\left( R^{n-1}_N (x) \right), \quad n \geq 2, \label{1.4}
\end{equation}
with $R_{N}^0 (x) = x$ and
\begin{equation}
a_1:=a_1(x) = \left\{\begin{array}{lll}
\left\lfloor \frac{N}{1-x} \right\rfloor & \hbox{if} & x \neq 1, \\
\infty & \hbox{if} & x = 1.
\end{array} \right. \label{1.5}
\end{equation}

These transformations belong to a wider one parameter family of interval maps of the form
$T_u (x) :=  \frac{1}{u(1-x)} - \lfloor \frac{1}{u(1-x)} \rfloor$, where $u>0$, $x \in [0, 1)$.
As the parameter varies in $(0, 4)$ there is a viable theory of a one parameter family of continued fractions, which fails when $u \geq 4$.
Named \textit{$u$-backward continued fractions}, they possess some attractive properties which are not shared by the regular continued fractions.
A natural question was whether the dynamical system given by the maps $T_u$ admits an absolutely continuous invariant measure.
Grochenig and Haas showed in \cite{Grochenig&Haas1-1996} that the invariant measure for $T_u$ is finite if and only if $0 < u < 4$ and
$u_q \neq 4\cos^2 \frac{\pi}{q}$, $q=3, 4, \ldots$.
They have identified that for certain values of $u$ (for example, if $u=1/N$ for positive integers $N \geq 2$) $R_N := T_{1/N}$ has a unique absolutely continuous invariant measure
\begin{equation}
\rho_N (A) :=
\frac{1}{\log \left(\frac{N}{N-1}\right)} \int_{A} \frac{\mathrm{d}x}{x+N-1}, \quad A \in {\mathcal{B}}_{[0,1]}, \label{1.3}
\end{equation}
where $\mathcal{B}_{[0,1]}$ denotes the $\sigma$-algebra of all Borel subsets of $[0,1]$.

It was proved in \cite{Grochenig&Haas1-1996} that the dynamical system $\left([0,1], R_N, \rho_N \right)$ is ergodic.
Using the ergodicity of $R_N$ and Birkhoff's ergodic theorem \cite{Dajani&Kraaikamp-2002} a number of results were obtained in
\cite{Grochenig&Haas1-1996, Haas&Molnar-2004}.
It should be stressed that the ergodic theorem does not yield any information on the convergence rate in the Gauss-Kuzmin problem that amounts to the asymptotic behaviour of $\mu \left( R^{-n}_N \right)$ as $n \rightarrow \infty$, where $\mu$ is an arbitrary probability measure on ${\mathcal{B}}_{[0,1]}$. So that a Gauss-Kuzmin theorem is needed.

Recently, in \cite{Sebe&Lascu-2018} we proved a version of a Gauss-Kuzmin theorem.
Using the natural extension for R\'enyi-type continued fraction expansions, we obtained an infinite-order-chain representation of the sequence of the incomplete quotients of these expansions.
Together with the ergodic behaviour of a certain homogeneous random system with complete connections this allowed us to solve a variant of the Gauss-Kuzmin problem.
Following the treatment in the case of the regular continued fractions \cite{IK-2002}, the Gauss-Kuzmin-L\'evy problem for the transformation $R_N$ can be approached in terms of the associated Perron-Frobenius operator.
In Section 2 we focus our study on the Perron-Frobenius operator under the invariant measure induced by the limit distribution function.
In Section 3 we use a Wirsing-type approach \cite{W-1974} to get close to the optimal convergence rate.
By restricting the domain of the Perron-Frobenius operator of $R_N$ under its invariant measure $\rho_N$ to the Banach space of functions which have a continuous derivative on $[0, 1]$, we obtain upper and lower bounds of the error which provide a refined estimate of the convergence rate.
The last section gives interesting numerical calculations.


\section{Operator-theoretical treatment}
Let $\left([0, 1],{\mathcal B}_{[0, 1]}, R_N, \rho_N \right)$ be as in Section 1.
In the sequel, we derive its Perron-Frobenius operator.
Let $\mu$ be a probability measure on $\left([0, 1], {\mathcal{B}}_{[0, 1]}\right)$
such that $\mu\left(R_N^{-1}(A)\right) = 0$ whenever $\mu(A) = 0$ for
$A \in {\mathcal{B}}_{[0, 1]}$, i.e., the transformation $R_N$ is $\mu$\textit{-non-singular}.
For example, this condition is satisfied if $R_N$ is $\mu$-preserving, that is,
$\mu R_N^{-1} = \mu$.

The \textit{Perron-Frobenius operator} $U$ associated with $R_N$ is defined as the bounded linear and positive operator
on the Banach space
$L^1([0,1], \mu) := \{f: [0,1] \rightarrow \mathbb{C} : \int^{1}_{0} |f |\,d\mu<\infty\}$ which satisfies
\begin{equation}
\int_{A}U f\, d\mu = \int_{R_N^{-1}(A)}f \,d\mu \quad \mbox{for all }f \in L^1([0, 1], \mu),\, A \in {\mathcal{B}_{[0, 1]}}, \label{2.1}
\end{equation}
or, equivalently
\begin{equation} \label{2.2}
\int_{0}^{1}g U f \,d\mu = \int_{0}^{1} (g \circ R_N) f\, d\mu \quad \mbox{ for all } f \in L^1([0, 1], \mu) \mbox{ and } g \in L^{\infty}([0, 1]).
\end{equation}
Here $L^{\infty}([0, 1])$ denotes the Banach space of $\lambda$-essentially bounded functions defined on $[0, 1]$, where $\lambda$ is the Lebesgue measure.
The existence and uniqueness of $Uf$ follows from the Radon-Nikodym theorem.
This implies also the existence and uniqueness of $U$.

In particular, the Perron-Frobenius operator $U$ of $R_N$ under the Lebesgue measure is given as follows (\cite{BG-1997}, p.86):
\begin{equation}
Uf(x)=\frac{d}{dx}\int_{R_N^{-1}([0,x])}f(t) \,dt = \sum_{t \in R_N^{-1}(x)}\frac{f(t)}{\left|(R_N(t))'\right|} \mbox{  } \mbox{  } \mbox{a.e. in } [0, 1]. \label{2.3}
\end{equation}
The following property is useful to prove the next proposition.

There exists $f \in L^1([0, 1], \mu)$ such that $f \geq 0$ and $Uf = f$ a.e. if and only if
$R_N$ preserves the measure $\nu$ which is defined as $\nu(A) := \int_{A} f d\mu$ for $A \in \mathcal{B}_{[0, 1]}$.
In particular, $U1=1$ if and only if $R_N$ is $\mu$-preserving (\cite{BG-1997}, p.80).

\begin{proposition}
The measure $\rho_N$ in $(\ref{1.3})$ is invariant under $R_N$ in $(\ref{1.2})$.
\end{proposition}
\begin{proof}
Let $\nu_N$ the density measure of $\rho_N$, i.e, $\nu_N(x):= \frac{1}{\log \left(\frac{N}{N-1}\right)} \frac{1}{x+N-1}$.
From above, it is sufficient to show that the function $\nu_N$ is an eigenfunction of the Perron-Frobenius operator of $R_N$ with the eigenvalue 1:
\begin{equation}\label{2.04}
U \nu_N(x) = \sum_{t \in R_N^{-1}(x)}\frac{\nu_N(t)}{\left|(R_N(t))'\right|}.
\end{equation}
First, we note that $R_N^{-1}(x) = \left\{ 1 - \frac{N}{x+i}: i \geq N, x \in [0, 1] \right\}$.
Thus,
\begin{eqnarray*}
U \nu_N(x)  &=& \sum_{i \geq N}\frac{N}{(x+i)^2} \nu_N \left( 1 - \frac{N}{x+i} \right) \\
            &=& \frac{1}{\log \left(\frac{N}{N-1}\right)} \sum_{i \geq N} \frac{1}{(x+i)(x+i-1)} \\
            &=& \frac{1}{\log \left(\frac{N}{N-1}\right)} \frac{1}{x+N-1} = \nu_N(x).
\end{eqnarray*}

\end{proof}

\begin{proposition} \label{prop.2.1}
Let $\left([0, 1],{\mathcal B}_{[0, 1]}, R_{N}, \rho_{N} \right)$ be as in Section 2, and let $U$ denote its Perron-Frobenius operator.
Then the following holds:
\begin{enumerate}
\item[(i)]
The following equation holds:
\begin{equation}
Uf(x) = \sum_{i \geq N} P_{N,i}(x)\,f\left(u_{N,i}(x)\right), \quad
f \in L^1([0, 1], \rho_{N}), \label{2.4}
\end{equation}
where $P_{N,i}$ and $u_{N,i}$ are functions defined on $[0, 1]$ by:
\begin{equation}\label{2.5}
P_{N,i}(x) := \frac{x+N-1}{(x+i)\,(x+i-1)}
\end{equation}
and
\begin{equation}\label{2.6}
u_{N,i}(x) := 1 - \frac{N}{x+i}.
\end{equation}
\item[(ii)]
Let $\mu$ be a probability measure on $\left([0, 1],{\mathcal{B}}_{[0, 1]}\right)$
such that $\mu \ll \lambda$, i.e., $\mu$ is absolutely continuous with respect to
the Lebesgue measure $\lambda$ and let $h := \mathrm{d}\mu / \mathrm{d} \lambda$ a.e. in $[0, 1]$.
Then for any $n \in \mathbb{N}_+:=\{1, 2, 3, \ldots\}$ and $A \in {\mathcal{B}}_{[0, 1]}$,
we have
\begin{equation}
\mu \left(R_{N}^{-n}(A)\right)
= \int_{A} U^nf(x) \mathrm{d}\rho_{N}(x) \label{2.7}
\end{equation}
where $f(x):= \left(\log\left(\frac{N}{N-1}\right)\right) (x+N-1) h(x)$ for
$x \in [0, 1]$.
\end{enumerate}
\end{proposition}
\begin{proof}
(i) Let $R_{N,i}$ denote the restriction of $R_N$ to the subinterval
$I(i):=\left(1-\frac{N}{i}, 1-\frac{N}{i+1}\right]$, $i \geq N$, that is,
\begin{equation}
R_{N,i}(x) = \frac{N}{1-x} - i, \quad x \in I(i). \label{2.8}
\end{equation}
Let $C(A):=\left(R_{N}\right)(A)$ and $C_{i}(A):=\left(R_{N,i}\right)^{-1}(A)$ for
$A \in {\mathcal B}_{[0, 1]}$.
Since $C(A)=\bigcup_{i}C_i(A)$ and $C_i\cap C_j$ is a null set when $i \neq j$,
we have
\begin{equation}
\int_{C(A)} f \,\mathrm{d} \rho_N = \sum_{i \geq N} \int_{C_i(A)}f\, \mathrm{d} \rho_N,
\quad
f \in L^1([0, 1], \rho_N),\,A \in {\mathcal{B}}_{[0, 1]}. \label{2.9}
\end{equation}
For any $i \geq N$, by the change of variable
$x = \left(R_{N,i}\right)^{-1}(y) = u_{N,i}(y)$,
we successively obtain
\begin{eqnarray}
\int_{C_i(A)}f(x) \,\rho_N(\mathrm{d}x) &=& \left(\log\left(\frac{N+1}{N}\right)\right)^{-1} \int_{C_i(A)} \frac{f(x)}{x+N-1}\,\mathrm{d}x \nonumber \\
\nonumber \\
&=& \left(\log\left(\frac{N+1}{N}\right)\right)^{-1} \int_{A} \frac{1}{(y+i)(y+i-1)}f\left(u_{N,i}(y)\right) \mathrm{d}y \nonumber \\
\nonumber \\
&=& \int_{A} P_{N,i}(y)\, f\left(u_{N,i}(y)\right)\,\rho_N (\mathrm{d}y). \label{2.10}
\end{eqnarray}
Now, (\ref{2.4}) follows from (\ref{2.9}) and (\ref{2.10}).

\noindent
(ii) We will use mathematical induction.
For $n=0$, the equation (\ref{2.7}) holds by definitions of $f$ and $h$.
Assume that (\ref{2.7}) holds for some $n \in \mathbb{N}$.
Then
\begin{equation} \label{5.7}
\mu \left(R_N^{-(n+1)}(A)\right) =
\mu \left(R_N^{-n}\left(R_N^{-1}(A)\right)\right)
= \int_{C(A)} U^n f(x)\,\rho_N(\mathrm{d}x),
\end{equation}
and by definition, we have
\begin{equation} \label{5.8}
\int_{C(A)} U^n f(x) \,\rho_N(\mathrm{d}x) = \int_{A} U^{n+1} f(x) \,\rho_N(\mathrm{d}x).
\end{equation}
Therefore,
\begin{equation} \label{5.9}
\mu \left(R_N^{-(n+1)}(A)\right)
= \int_{A} U^{n+1} f(x)\rho_N(\mathrm{d}x)
\end{equation}
which ends the proof.
\end{proof}
\begin{remark} \label{rem.2.3}
In hypothesis of Proposition \ref{prop.2.1}(ii) it follows that
\begin{equation}
\mu(R_{N}^{-n}(A)) - \rho_{N}(A) = \int_{A}(U^{n}f(x)-1)\mathrm{d}\rho_{N}(x), \label{2.15}
\end{equation}
for any $n \in \mathbb{N}$ and $A \in {\mathcal{B}}_{[0, 1]}$, where
$f(x):= \left(\log\left(\frac{N}{N-1}\right)\right) (x+N-1) h(x)$, $x \in [0, 1]$.
The last equation shows that the asymptotic behavior of $\mu(R_{N}^{-n}(A)) - \rho_{N}(A)$ as $n \rightarrow \infty$ is given by the asymptotic behavior of the $n$-th power of the Perron-Frobenius operator $U$ on $L^1([0, 1], \rho_N)$.
\end{remark}

\section{Near-optimal solution to Gauss-Kuzmin-L\'evy problem}

In this section we develop a Wirsing-type approach \cite{W-1974} to obtain a solution to Gauss-Kuzmin-L\'evy problem in Theorem \ref{th.3.4}.

Let $\mu$ be a probability measure on ${\mathcal{B}}_{[0, 1]}$ such that $\mu \ll \lambda$.
For any $n \in \mathbb{N}$ put $F^n_N (x) = \mu \left( R^n_N < x \right)$, $x \in [0, 1]$, where $R^0_N$ is the identity map.
As $\left( R^n_N < x \right) = R^{-n}_N ((0, x))$, by Proposition 2.2 (ii) we have
\begin{equation}
F^n_N (x) = \int_{0}^{x} \frac{U^nf^0_N(u)}{x+N-1} \mathrm{d}u, \quad n \in \mathbb{N}, \label{3.1}
\end{equation}
where $f^0_N(x):= (x+N-1) \left( F_N^0\right)'(x)$, $x \in [0, 1]$, where $\left( F_N^0\right)' = \mathrm{d} \mu / \mathrm{d} \lambda$.

We will assume that $\left( F_N^0\right)' \in C^1([0, 1])$,
the collection of all functions $f: [0, 1] \to {\mathbb{C}}$ which have a continuous derivative.
So, we study the behavior of $U^n$ as $n \to \infty$, assuming that the domain of $U$ is $C^1 ([0, 1])$.

Let $f \in C^1 ([0, 1])$. Then the series (\ref{2.4}) can be differentiated
term-by-term, since the series of derivatives is uniformly convergent.
Next, since
\begin{equation*}
P_{N,i}(x) = \left( \frac{i+1-N}{x+i} - \frac{i-N}{x+i-1} \right),
\end{equation*}
we get
\begin{eqnarray}
(U f)'(x) &=& \sum_{i \geq N}
\left\{
(P_{N,i})'(x) f\left(u_{N,i}(x)\right) + P_{N,i}(x) f'\left(u_{N,i}(x)\right) \left(u_{N,i}\right)'(x)
\right\} \nonumber \\
&=& \sum_{i \geq N}
\left\{
\left( \frac{i-N}{(x+i-1)^2} - \frac{i+1-N}{(x+i)^2} \right) f\left(u_{N,i}(x)\right) + P_{N,i}(x) f'\left(u_{N,i}(x)\right) \frac{N}{(x+i)^2}
\right\} \nonumber \\
&=& \sum_{i \geq N}
\left\{
\frac{i+1-N}{(x+i)^2}\left[ f\left(u_{N,i+1}(x)\right) - f\left(u_{N,i}(x)\right) \right]+P_{N,i}(x) f'\left(u_{N,i}(x)\right) \frac{N}{(x+i)^2}
\right\},  \label{5.3}
\end{eqnarray}
for any $x \in [0, 1]$.
Thus, we can write
\begin{equation}
(U f)' = - V f', \quad f \in C^1 ([0, 1]), \label{3.3}
\end{equation}
where $V : C([0, 1]) \to C([0, 1])$ is defined by
\begin{equation}
V g(x) = - \sum_{i \geq N}
\left\{
\frac{i+1-N}{(x+i)^2} \int^{u_{N,i+1}(x)}_{u_{N,i}(x)} g(u)\mathrm{d}u
+ \frac{N(x+N-1)}{(x+i-1)(x+i)^3} g\left(u_{N,i}(x)\right)
\right\}  \label{3.4}
\end{equation}
with $g \in C([0, 1])$ and $x \in [0, 1]$.
Clearly,
$(U^n f)' = (-1)^n V^n f'$, $n \in {\mathbb{N}}_+, f \in C^1([0, 1])$.
We are going  to show that $V^n$ takes certain  functions into functions with very small values when $n \in {\mathbb{N}}_+$ is large.
\begin{proposition} \label{prop3.1}
For a fixed integer $N \geq 2$ there are positive constants $v_N < w_N < 1$ and a real-valued non-positive function $\varphi_N \in C([0, 1])$ such that
\begin{equation}
v_N (-\varphi_N) \le V \varphi_N \le w_N (-\varphi_N). \label{3.5}
\end{equation}
\end{proposition}
\begin{proof}
For $\mathbb{R}_+:=\{x \in \mathbb{R}: x \geq 0 \}$, let $h_N : \mathbb{R}_+ \rightarrow \mathbb{R}$ be a continuous bounded function such that $\lim_{x \rightarrow \infty} h_N(x) < \infty$.
We look for a function $g_N : (0, 1] \rightarrow \mathbb{R}$ such that $U g_N = h_N$, assuming that the equation
\begin{equation}
U g_N(x) = \sum_{i \geq N} P_{N,i}(x) g_N\left(u_{N,i}(x)\right) = h_N(x) \label{3.6}
\end{equation}
holds for $x \in {\mathbb{R}}_+$.
By reducing the terms of the series involved (\ref{3.6}) yields
\begin{equation}
\frac{h_N(x)}{x + N -1} - \frac{h_N(x+1)}{x+N} =
\frac{1}{(x+N-1)(x+N)} g_N \left(\frac{x}{x+N}\right), \quad x \in {\mathbb{R}}_+. \label{3.7}
\end{equation}
Hence
\begin{equation}
g_N(u) = \frac{N}{1-u} h_N\left(\frac{Nu}{1-u}\right) - \left( \frac{N}{1- u} -1 \right) h_N\left(\frac{Nu}{1-u}+1 \right), \quad u \in (0, 1], \label{3.8}
\end{equation}
and we indeed have $U g_N = h_N$ since
\begin{eqnarray}
U g_N(x) &=& \sum_{i \geq N} \frac{x+N-1}{(x+ i-1 )(x+i )} g_N\left(1 - \frac{N}{x+i }\right) \nonumber \\
&=& (x+N-1) \sum_{i \geq N} \left( \frac{1}{x+i-1} - \frac{1}{x+i} \right) g_N\left(1 - \frac{N}{x+i }\right) \nonumber \\
&=& (x+N-1) \sum_{i \geq N} \left( \frac{h_N(x+i-N)}{x+i-1} - \frac{h_N(x+i+1-N)}{x+i} \right) \nonumber \\
&=& (x+N-1) \left( \frac{h_N(x)}{x+N-1} - \lim_{i \to \infty} \frac{h_N(x+i+1-N)}{x+i} \right) = h_N(x), \quad x \in {\mathbb{R}}_+. \label{3.9}
\end{eqnarray}

In particular, for any fixed $t_N \in [0, 1]$ we consider the function
$h_{N,t_N} : {\mathbb{R}}_+ \rightarrow \mathbb{R}$ defined as
\begin{equation}
h_{N,t_N} (x) = \frac{1}{e_N x + t_N +1}, \quad x \in {\mathbb{R}}_+ \label{3.10}
\end{equation}
where the coefficient $e_N$ will be specified later.
By the above, the function $g_{N,t_N} : (0, 1] \rightarrow \mathbb{R}$ defined as
\begin{eqnarray}
g_{N,t_N}(x) &=& \frac{N}{1-x} h_{N,t_N}\left(\frac{Nx}{1-x}\right) - \left( \frac{N}{1-x}-1 \right) h_{N,t_N}\left(\frac{Nx}{1-x}+1 \right) \nonumber \\
       &=& \frac{N}{e_N Nx+(t_N+1)(1-x)} - \frac{N-1+x}{e_N (Nx+1-x) + (t_N+1)(1-x)} \label{3.11}
\end{eqnarray}
for any $x \in (0, 1]$ satisfies
\begin{equation}
U g_{N,t_N}(x) = h_{N,t_N}(x), \quad x \in [0, 1]. \label{3.12}
\end{equation}
Setting
\begin{equation}
\varphi_{N,t_N}(x) := (g_{N,t_N})'(x) = \frac{-N(e_N N -t_N-1)}{\left\{e_N Nx+(t_N+1)(1-x)\right\}^2} -
\frac{N\left\{2e_N-(e_N N-t_N-1)\right\}}{\left\{e_N (Nx+1-x) + (t_N+1)(1-x)\right\}^2} \label{3.13}
\end{equation}
we have
\begin{equation}
V \varphi_{N,t_N}(x) = - (U g_{N,t_N})'(x) = - (h_{N,t_N})'(x) = \frac{e_N}{(e_N x+t_N+1)^2}, \quad x \in [0, 1]. \label{3.14}
\end{equation}
Since $g_{N,t_N}$ is a decreasing function it follows that $\varphi_{N,t_N}(x)<0$, $x \in [0, 1]$.
Also, $V$ is a linear operator that takes non-positive functions into positive functions. Therefore, $V \varphi_{N,t_N}(x)>0$, $x \in [0, 1]$.

We choose $t_N$ by asking that
$
(\varphi_{N,t_N} / V \varphi_{N,t_N})(0) = (\varphi_{N,t_N} / V \varphi_{N,t_N})(1).
$
Since
\begin{equation}
(\varphi_{N,t_N} / V \varphi_{N,t_N})(0) = \frac{(t_N+1)^2}{e_N} \left\{ \frac{-N(e_N N-t_N-1)}{(t_N+1)^2} -
\frac{N \left\{2e_N-(e_N N-t_N-1)\right\}}{(e_N+t_N+1)^2} \right\} \label{3.15}
\end{equation}
and
\begin{equation}
(\varphi_{N,t_N} / V \varphi_{N,t_N})(1) = \frac{-2(e_N+t_N+1)^2}{e_N^2 N} \label{3.16}
\end{equation}
this amounts to the equation
\begin{equation}
H_N(t_N) = 2 (t_N+e_N+1)^4 +e_N^3 N^2 (1-2N)(t_N+1)-e_N^4 N^3 = 0. \label{3.17}
\end{equation}
We choose the coefficient $e_N$ such that the equation
$H_N (x) = 0$, $x \in [0, 1]$, yields a unique solution $t_N \in [0, 1]$.
Asking that
\begin{equation}
H_N (0) < 0, \quad H_N (1) > 0,
\quad \hbox{and} \quad \frac{\mathrm{d}H_N}{\mathrm{d}t_N}> 0,  \label{3.18}
\end{equation}
we may determine $e_N$ (see Appendix).
For this unique acceptable solution $t_N \in [0, 1]$ the function
$\varphi_{N,t_N} / V \varphi_{N,t_N}$ attains its minimum equal to ${-2(e_N+t_N+1)^2}/{\left(e_N^2 N \right)}$ at $x=0$ and $x=1$,
and has a maximum $m(t_N)=\varphi_{N,t_N} / V \varphi_{N,t_N}(x_{max}) < 0$.
It follows that
\begin{equation*}
\frac{-2(e_N+t_N+1)^2}{e_N^2 N} \leq \frac{\varphi_{N,t_N}}{V \varphi_{N,t_N}} \leq m(t_N).
\end{equation*}
Since $\displaystyle \frac{\varphi_{N,t_N}}{V \varphi_{N,t_N}} < 0$ we get
\begin{equation*}
-m(t_N) \leq \frac{-\varphi_{N,t_N}}{V \varphi_{N,t_N}} \leq \frac{2(e_N+t_N+1)^2}{e_N^2 N}.
\end{equation*}
Therefore,
\begin{equation*}
\frac{e_N^2 N}{2(e_N+t_N+1)^2} (-\varphi_{N,t_N}) \leq V \varphi_{N,t_N} \leq -\frac{1}{m(t_N)}(-\varphi_{N,t}).
\end{equation*}
It follows that for $\varphi_N = \varphi_{N,t_N}$ we have
\[
v_N (-\varphi_N) \leq V \varphi_N \leq w_N (-\varphi_N),
\]
where
\begin{equation*}
v_N = \frac{e_N^2 N}{2(e_N+t_N+1)^2} \quad \hbox{and} \quad w_N = -\frac{1}{m(t_N)}.
\end{equation*}
\end{proof}

\begin{remark} \label{rem3.2}
By $(\ref{3.5})$ we successively get

\begin{eqnarray*}
  &&v_N^2 (-\varphi_N) \leq -V^2 \varphi_N \leq w_N^2 (-\varphi_N) \\
  &&v_N^3 (-\varphi_N) \leq V^3 \varphi_N  \leq w_N^3 (-\varphi_N) \\
  &&\vdots \\
  &&v_N^n (-\varphi_N) \leq (-1)^{n+1} V^n \varphi_N \leq w_N^n (-\varphi_N), \quad n \in \mathbb{N}_+.
\end{eqnarray*}
\end{remark}

\begin{corollary} \label{cor3.3}
Let $f^0_{N} \in C^1([0, 1])$ such that $(f^0_{N})'>0$.
Put
\[
\alpha_{N} = \displaystyle \min_{x \in [0, 1]} \frac{- \varphi_{N}(x)}{(f^0_{N})'(x)}  \,
\mbox{ and } \, \beta_{N} = \displaystyle \max_{x \in [0, 1]} \frac{- \varphi_{N} (x)}{(f^0_{N})'(x)}.
\]
Then
\begin{equation}
\frac{\alpha_{N}}{\beta_{N}} v^n_{N} (f^0_{N})'(x) \leq (-1)^n V^n (f^0_{N})'(x)
\leq \frac{\beta_{N}}{\alpha_{N}} w^n_{N} (f^0_{N})'(x), \quad n \in {\mathbb{N}}_+, \, x \in [0, 1]. \label{3.19}
\end{equation}
\end{corollary}
\begin{proof}
Noting that $\alpha_{N} (f^0_{N})'(x) \leq (-\varphi_{N})(x) \leq \beta_{N} (f^0_{N})'(x)$ and using Remark \ref{rem3.2}
we can write
\begin{eqnarray*}
\frac{\alpha_{N}}{\beta_{N}} v^n_{N} (f^0_{N})'(x) &\leq& \frac{1}{\beta_N} v^n_{N} (-\varphi_{N})(x) \leq \frac{1}{\beta_N} (-1)^{n+1} V^n \varphi_{N}(x) \\
&\leq& (-1)^n V^n (f^0_{N})'(x) \leq \frac{(-1)^{n+1}}{\alpha_N} V^n \varphi_N(x) \\
&\leq& \frac{w_N^n}{\alpha_N} (-\varphi_N)(x) \leq \frac{\beta_{N}}{\alpha_{N}} w^n_{N} (f^0_{N})'(x), \quad n \in {\mathbb{N}}_+,
\end{eqnarray*}
which shows that (\ref{3.19}) holds.
\end{proof}

\begin{theorem} \label{th.3.4}
Let $f^0_{N} \in C^1([0, 1])$ such that $(f^0_{N})' > 0$
and let $\mu$ be a probability measure on ${\mathcal B}_{[0, 1]}$
such that $\mu \ll \lambda$.
For any $n \in {\mathbb{N}}_+$ and $x \in [0, 1]$ we have
\begin{eqnarray}
&&\left(\log \left(\frac{N}{N-1}\right)\right)^2\, \cdot \frac{N}{2} \cdot \frac{\alpha_{N}}{\beta_{N}} \min_{x \in [0, 1]} (f^0_{N})' (x) \cdot v^n_{N} G_{N} (x) (1 - G_{N}(x)) \nonumber \\
&&\leq |\mu (R_N^n < x) - G_{N} (x)| \nonumber \\
&&\leq
\left(\log \left(\frac{N}{N-1}\right)\right)^2\, \cdot \frac{N}{2} \cdot \frac{\beta_{N}}{ \alpha_{N}} \max_{x \in [0, 1]} (f^0_{N})'(x) \cdot w^n_{N} G_{N} (x) (1 - G_{N}(x)) \qquad \qquad
\end{eqnarray}
where $\alpha_{N}$, $\beta_{N}$, $v_{N}$ and $w_{N}$ are defined in Proposition $\ref{prop3.1}$ and Corollary $\ref{cor3.3}$, and
\begin{equation}
G_{N} (x) = \frac{1}{\log \left(\frac{N}{N-1}\right)} \log \left( \frac{x+N-1}{N-1} \right). \label{3.21}
\end{equation}
\end{theorem}

\begin{proof}
For any $n \in {\mathbb{N}}$ and $x \in [0, 1]$ set
$d_n (G_{N} (x)) = \mu (R^n_{N} < x) - G_{N} (x)$, with $G_N$ as in (\ref{3.21}).
Then by (\ref{3.1}) we have
\begin{equation*}
d_n (G_{N} (x)) = \int^x_0 \frac{U^n f^0_{N}(u)}{u+N-1} \mathrm{d}u - G_{N} (x).
\end{equation*}
Differentiating twice with respect to $x$ yields
\begin{equation*}
d'_n (G_{N}(x)) \frac{1}{\log \left(\frac{N}{N-1}\right)} \frac{1}{x+N-1} =
\frac{U^n f^0_{N}(x)}{x+N-1} - \frac{1}{\log \left(\frac{N}{N-1}\right)} \frac{1}{x+N-1},
\end{equation*}
\begin{equation}
(U^n f^0_{N})'(x) = d''_n (G_{N}(x)) \left(\frac{1}{\log \left(\frac{N}{N-1}\right)}\right)^2 \frac{1}{x+N-1}, \quad n \in {\mathbb{N}},
x \in [0, 1].
\end{equation}
Hence by (\ref{3.3}) we have
\begin{eqnarray*}
d''_n (G_{N} (x)) = \left(\frac{1}{\log \left(\frac{N}{N-1}\right)}\right)^2 (x+N-1) \left(U^n f^0_{N} \right)'(x) \\
= (-1)^n \left(\frac{1}{\log \left(\frac{N}{N-1}\right)}\right)^2 (x+N-1) V_N^n (f^0_{N})'(x),
\end{eqnarray*}
for any $n \in {\mathbb{N}}$, $x \in [0, 1]$.
Since $d_n (0) = d_n (1) = 0$, a well-known interpolation formula yields
\begin{equation}
d_n (x) = - \frac{x(1-x)}{2} d''_n (\xi), \quad n \in {\mathbb{N}}, x \in [0, 1],
\end{equation}
for a suitable $\xi = \xi (n,x) \in [0, 1]$.
Therefore
\begin{eqnarray*}
\mu (R^n_{N} < x) - G_{N} (x) = - \frac{G_{N} (x) (1- G_{N}(x))}{2} d''_n (G_N(\xi_N)) \\
=(-1)^{n+1} \left(\frac{1}{\log \left(\frac{N}{N-1}\right)}\right)^2 (\xi_N +N-1) V_N^n (f^0_{N})'(\xi_N) \frac{G_{N} (x) (1- G_{N}(x))}{2} \\
\leq  (-1)^{n+1} \left(\frac{1}{\log \left(\frac{N}{N-1}\right)}\right)^2 \frac{N}{2} V_N^n (f^0_{N})'(\xi_N) {G_{N} (x) (1- G_{N}(x))}
\end{eqnarray*}
for any $n \in {\mathbb{N}}$ and $x \in [0, 1]$, and another suitable
$\xi_{N} = \xi_{N} (n,x) \in [0, 1]$.
The result stated follows now from Corollary \ref{cor3.3}.
\end{proof}

\section{Final remarks}

To conclude, we use the values obtained in the Appendix.

Let us consider the case $N=3$.
The equation $H_{3}(x)=0$, with $e_{3}=0.8956735$,
has as unique acceptable solution $t_{3}=0.4999967$.
For this value of $t_{3}$ the function $\varphi_{t_{3}}/V \varphi_{t_{3}}$
attains its minimum equal to $-4.76939599403913$ at $x=0$ and $x=1$,
and has a maximum $m(t_3)=(\varphi_{t_{3}}/V\varphi_{t_{3}})(0.423325998187593)=-4.62762782434937$.
It follows that upper and lower bounds of the convergence rate are respectively
$O(w_{3}^n)$ and $O(v_{3}^n)$ as $n \to \infty$, with $v_{3}>0.20967015556054$ and $w_{3}<0.216093436628214$.

Let us consider the case $N=5$.
The equation $H_{5}(x)=0$, with $e_{5}=0.4088150$,
has as unique acceptable solution $t_{5}=0.5000000$.
For this value of $t_{5}$ the function $\varphi_{t_{5}}/V \varphi_{t_{5}}$
attains its minimum equal to $-8.72035227508647$ at $x=0$ and $x=1$,
and has a maximum $m(t_5)=(\varphi_{t_{5}}/V\varphi_{t_{5}})(0.457198440687485)=-8.64358828233062$.
It follows that upper and lower bounds of the convergence rate are respectively
$O(w_{5}^n)$ and $O(v_{5}^n)$ as $n \to \infty$, with $v_{5}>0.114674266412028$ and $w_{5}<0.115692692356046$.

Finally, let us consider the case $N=100$.
The equation $H_{100}(x)=0$, with $e_{100}=0.0152027$,
has as unique acceptable solution $t_{100}=0.4999998$.
For this value of $t_{N}$ the function $\varphi_{t_{100}}/V \varphi_{t_{100}}$
attains its minimum equal to $-198.668858764086$ at $x=0$ and $x=1$,
and has a maximum $m(t_{100})=(\varphi_{t_{100}}/V\varphi_{t_{100}})(0.49804751660470764)=-198.66555309796482$.
It follows that upper and lower bounds of the convergence rate are respectively
$O(w_{100}^n)$ and $O(v_{100}^n)$ as $n \to \infty$, with $v_{100}>0.00503350150708559$ and $w_{100}<0.00503358526129032$.

\begin{center}

\begin{tabular}{| c | c | c | c | c | c |}
\hline
N=3           & $v_3>0.20967015556054$          &  $w_3<0.216093436628214$  \\ \hline
N=5           & $v_5>0.114674266412028$         &  $w_5<0.115692692356046$ \\ \hline
\, \, N=100   & $v_{100}>0.00503350150708559$   &  $w_{100}<0.00503358526129032$ \\ \hline
\end{tabular}
\end{center}

\section{Appendix}
Imposing conditions (\ref{3.18}) and using MATHEMATICA we obtain
\begin{center}

\begin{tabular}{| c | c | c | c | c | c |}
\hline
N & $e_{N}$ & $t_{N}$ \\ \hline\hline
2 & 2.1780250 &  0.4999997  \\ \hline
3 & 0.8956735 &  0.4999967 \\ \hline
4 & 0.5616365 &  0.5000001 \\ \hline
5 & 0.4088150 &  0.5000000 \\ \hline
10 & 0.1730660 & 0.5000007  \\ \hline
15 & 0.109754 &	 0.4999972   \\ \hline
20 & 0.080357 &	 0.5000078   \\ \hline
25 & 0.063380 &	 0.4999999  \\ \hline
30 & 0.052326 &	 0.5000149  \\ \hline
35 & 0.044554 &	 0.4999865  \\ \hline
40 & 0.038793 &	 0.4999972  \\ \hline
45 & 0.034351 &	 0.4999945  \\ \hline
50 & 0.030822 &	  0.5000046  \\ \hline
100 & 0.0152027 & 0.4999998 \\ \hline
1000 & 0.00150201	& 0.5000073 \\ \hline
10000 & 0.00015002	& 0.4999999 \\ \hline
\end{tabular}
\end{center}


\begin{thebibliography}{[01]}
%
\bibitem{BG-1997} Boyarsky, A., G\'ora, P., \textit{Laws of Chaos: Invariant Measures and Dynamical Systems in One Dimension}, Birkh\"auser, Boston, 1997.
%
\bibitem{Dajani&Kraaikamp-2002} Dajani, K., Kraaikamp, C., \textit{Ergodic Theory of Numbers}, The Carus Mathematical Monographs, Washington, 2002
%
\bibitem{Grochenig&Haas1-1996} Gr\"ochenig, K., Haas, A., \textit{Backward continued fractions, Hecke groups and invariant measures for transformations of the interval}, Ergodic Theory and Dynamical Systems \textbf{16}(6) (1996) 1241-1274.
%
\bibitem{Haas&Molnar-2004} Haas, A., Molnar, D., \textit{Metrical diophantine approximation for continued fraction like maps of the interval}, Transactions of the American Mathematical Society \textbf{356}(7) (2004) 2851-2870.
%
\bibitem{IK-2002} Iosifescu, M., Kraaikamp, C., \textit{Metrical Theory of Continued Fractions}, Kluwer Academic Publishers, Dordrecht, 2002.
%
%
%
%
\bibitem{Sebe&Lascu-2018} Sebe, G.I., Lascu, D., \textit{A dependence with complete connections approach to generalized R\'enyi continued fractions} (2018) submitted.
%
\bibitem{W-1974} Wirsing, E., \textit{On the theorem of Gauss-Kuzmin-L\'evy and a Frobenius-type theorem for function spaces}, Acta Arithmetica \textbf{24} (1974), 506-528.

\end{thebibliography}
\end{document}